\documentclass[10pt,reqno]{amsart}
\usepackage{float}
\usepackage{graphicx}
\usepackage{graphics}
\usepackage{subfigure}
\usepackage{mathrsfs}
\usepackage[colorlinks,linkcolor=black,anchorcolor=black,citecolor=black,hyperindex=black,CJKbookmarks=black]{hyperref}

\newtheorem{theorem}{Theorem}[section]
\newtheorem{lemma}[theorem]{Lemma}
\newtheorem{proposition}[theorem]{Proposition}

\theoremstyle{definition}
\newtheorem{definition}[theorem]{Definition}

\newtheorem{remark}[theorem]{Remark}
\newtheorem{corollary}[theorem]{Corollary}

\numberwithin{equation}{section}

\begin{document}

\title[]{Multifractal analysis of the convergence exponent in continued fractions}

\author {Lulu Fang}
\address{School of Mathematics, Sun Yat-sen University, Guangzhou, GD 510275, P.R.~China}
\email{fanglulu@mail.sysu.edu.cn}

\author {Kunkun Song}
\address{School of Mathematics and Statistics, Wuhan University, Wuhan, 430072, P. R. China}
\address{lama umr 8050, cnrs, université paris-est créteil, 61 avenue du général de gaulle, 94010, créteil cedex, france}
\email{songkunkun@whu.edu.cn}

%\thanks {* Corresponding author}

\begin{abstract}
Let $x \in [0,1)$ be a real number and denote its continued fraction expansion by $[a_1(x),a_2(x), a_3(x),\cdots]$.
The convergence exponent of these partial quotients is defined as
\[
\tau(x):= \inf\left\{s \geq 0: \sum_{n \geq 1} a^{-s}_n(x)<\infty\right\}.
\]
In this paper, we investigate some fundamental properties and multifractal analysis of the exponent $\tau(x)$.

\end{abstract}

\maketitle

\section{Introduction}

Let $\{b_n\}_{n \geq 1}$ be a sequence of positive integers. The convergence exponent \emph{(or exponent of convergence)} of $\{b_n\}_{n \geq 1}$ is defined as
\[
\tau:= \inf\left\{s \geq 0: \sum_{n \geq 1} b^{-s}_n<\infty\right\}=\sup\left\{t \geq 0: \sum_{n \geq 1} b^{-t}_n=\infty\right\}.
\]
It is well known that if $\{b_n\}_{n \geq 1}$ is non-decreasing, then the index $\tau$ can be calculated by using the following formula:
\begin{equation}\label{cal}
\tau = \limsup_{n \to \infty} \frac{\log n}{\log b_n},
\end{equation}
see Markushevich \cite[Theorem 10.2]{lesMar}. Such an index is arisen in serval different fields and plays an important role in depicting the fractal dimension of some sets.
It has been used by Besicovitch and Taylor \cite{BesTay54} to study the Hausdorff measure and Hausdorff dimension of a compact set whose complementary components are open intervals with length $\{b^{-1}_n\}_{n \geq 1}$;
by Dodson \cite{Dodson92} to characterize the Hausdorff dimension of certain sets associated with Diophantine approximations;
by Fan et al.\,\cite{FLMW10} to determine the Hausdorff dimension of Besicovitch-Eggleston sets in countable symbolic space;
by Hawkes \cite{lesHaw} to find the entropy dimension of some random sets determined by L\'{e}vy processes;
by Li et al.\,\cite{LSX13} to describe the Hausdorff dimension and hitting probabilities of random covering sets;
by Wang and Wu \cite{lesWW08A} to prove Hirst's conjecture on the Hausdorff dimension of certain sets of real numbers whose partial quotients tend to infinity and also lie in a fixed strictly increasing infinite sequence of integers.

In this paper, we will investigate the convergence exponent of partial quotients in continued fractions. Recall that each real number $x \in [0,1)$ admits a continued fraction expansion of the form
\[
x = \dfrac{1}{a_1(x) +\dfrac{1}{a_2(x) +\dfrac{1}{a_3(x)+ \ddots}}}:=[a_1(x), a_2(x), a_3(x), \cdots],
\]
where $a_1(x), a_2(x), a_3(x), \cdots$ are positive integer, called the \emph{partial quotients} of $x$.
Such an $x$ is irrational if and only if it has a unique and infinite continued fraction expansion,
see Iosifescu and Kraaikamp \cite{lesIK02} for more arithmetic and metric results of continued fractions.
Let $\tau(x)$ be the convergence exponent of the sequence of partial quotients of $x$. It is easy to see that the exponent $\tau$ takes values in $[0,\infty]$ and that $\tau(x) =0$ for any rational number $x \in [0,1)$. We are concerned with the fundamental properties (e.g., measurability, intermediate value property, continuity, etc) and multifractal analysis of the convergence exponent $\tau$.
This type of problem have been studied for many and various exponents related to continued fractions, for instance, see Jarn\'{\i}k \cite{Jar28} for the Diophantine approximation exponent,
Pollicott and Weiss \cite{lesPW99} for the Lyapunov exponent of Gauss map, Kesseb\"{o}hmer and Stratmann \cite{lesKS} for the Minkowski's question mark function, Fan et al.\,\cite{FLWW09} for the Khintchine exponent,  Nicolay and Simons \cite{lesNS} for the Cantor bijection, Jaffard and Martin \cite{lesJM} for the Brjuno function and so on.

For the real number $x\in [0,1)$ whose partial quotients are non-decreasing, by (\ref{cal}), we know that its convergence exponent $\tau(x)$ is described by the slow growth rate of partial quotients of $x$.
In this case, the multifractal analysis of $\tau$ is related to the set of real numbers for which their partial quotients slowly tends to infinity.
Actually, the fractal dimension of certain sets arising in continued fractions with some
restrictions on the growth rate of partial quotients has a long history, see Jarn\'{\i}k \cite{Jar28}, Good \cite{lesGood41}, Hirst \cite{lesHir73}, Cusick \cite{Cus90},
{\L}uczak \cite{lesLuc97}, Wang and Wu \cite{lesWW08A, lesWW08}, Fan et al.\,\cite{FLWW09, FLWW13},
Tong and Wang \cite{TW}, Jordan and Rams \cite{lesJR12}, Liao and Rams \cite{LR16},
just to mention a few.

The paper is organized as follows.
Section 2 is devoted to presenting our main results and their explanations. In Section 3, some basic results of continued fractions and useful lemmas for calculating the Hausdorff dimension of a fractal set are listed.
The proofs of our main results are given in Section 4.
Throughout this paper, we use $|\cdot|$ to denote the length of a subset of $[0,1)$, $\lfloor x\rfloor$ the largest integer smaller than $x$, $cl$ the closure of a set,  $\#$ the cardinality of a finite set, $\dim_{\rm H}$ the Hausdorff dimension and $\mathcal{H}^{s}$ the $s$-dimensional Hausdorff measure, respectively.

\section{Main results}
We first give some fundamental properties of the exponent $\tau:[0,1) \to [0,\infty]$.

\begin{theorem}\label{Fu}
The following results hold.\\
(i) The function $\tau: [0,1) \to [0,\infty]$ is Borel measurable.\\
(ii) For Lebesgue almost all $x \in [0,1)$, we have $\tau(x) =\infty$.\\
(iii) For any $0 \leq \alpha\leq \infty$, there exists an irrational number $x \in [0,1)$ such that $\tau(x) =\alpha$.\\
(iv) For any $0 \leq \alpha\leq \infty$, the level set $\{x \in [0,1): \tau(x) =\alpha\}$ is uncountable and dense in $[0,1)$.\\
(v) For any interval $I \subseteq [0,1)$, we have $\tau(I) = [0,\infty]$.\\
(vi) The function $\tau: [0,1)\to [0,\infty]$ is everywhere discontinuous.\\
\end{theorem}

In what follows, we turn to the multifractal analysis of the exponent $\tau$, namely the Hausdorff dimension of the level sets $\{x \in [0,1): \tau(x) =\alpha\}$ for all $\alpha \geq 0$.

\begin{theorem}\label{No}
For any $\alpha\geq0$, we have
\begin{equation*}
\dim_{\rm H} \big\{x \in [0,1): \tau(x) =\alpha\big\}=
\begin{cases}
1/2, &\ \ \ \ 0\leq \alpha<\infty;\cr
1, & \ \ \ \ \ \ \alpha=\infty.
\end{cases}
\end{equation*}
\end{theorem}

As a consequence of Theorem \ref{No} and Good \cite[Theorem 1]{lesGood41}, we have
\begin{corollary}\label{NoCor}
For any $\alpha \in [0,\infty)$, we have
\begin{equation*}
\dim_{\rm H} \big\{x \in [0,1): \tau(x) \leq \alpha\big\}=1/2.
\end{equation*}
\end{corollary}

\indent Let $\Lambda$ be the set of all irrationals $x \in [0,1)$ such that $\{a_n(x)\}_{n \geq 1}$ is non-decreasing and $a_n(x) \to \infty$ as $n \to \infty$.
The set $\Lambda$ is a Lebesgue null set but has Hausdorff dimension $1/2$, see Jordan and Rams \cite{lesJR12}.
We are interested in the multifractal analysis of the exponent $\tau$ over the fractal set $\Lambda$ and its spectrum function, namely the function of the Hausdorff dimension of $\{x \in \Lambda: \tau(x) =\alpha\}$ with respect to $\alpha$.

\begin{theorem}\label{tau}
For any $\alpha\geq0$, we have
\begin{equation*}
\dim_{\rm H} \big\{x \in \Lambda: \tau(x) =\alpha\big\}=
\begin{cases}
(1-\alpha)/2, &\ \ \ \ 0\leq \alpha\leq1;\cr
0, & \ \ \ \ \ \ \alpha>1.
\end{cases}
\end{equation*}
\end{theorem}

\begin{remark}
(i) We conclude from Theorem \ref{tau} that for any $\alpha\in[0,\infty)$,
\[
\dim_{\rm H} \{x \in \Lambda: \tau(x) \leq \alpha\} = 1/2.
\]
Indeed, the upper bound is trivial since the set $\Lambda$ is of Hausdorff dimension $1/2$. For the lower bound, it follows from Theorem \ref{tau} that $\dim_{\rm H} \{x \in \Lambda: \tau(x) =0\}=1/2$, which yields
$\dim_{\rm H} \{x \in \Lambda: \tau(x) \leq \alpha\} \geq 1/2$  for all $\alpha \geq 0$.\\
(ii) Theorem \ref{No} shows that there is no any multifractal phenomenon for the exponent $\tau$ in this case, i.e., the Hausdorff dimension of the set $\{x \in [0,1): \tau(x) =\alpha\}$ does not vary with the level $\alpha$. However, Theorem \ref{tau} shows that $\tau$ displays a multifractal phenomenon when it is restricted in $\Lambda$.\\
(iii)  Analysis similar to that in the proof of Theorem \ref{No} gives that
\[
\dim_{\rm H} \big\{x \in [0,1): \tau(x) =\alpha,~a_n(x) \to \infty\big\}=1/2.
\]
This means that the monotonicity of sequences $\{a_n(x)\}_{n \geq 1}$ in $\Lambda$ plays an essential role in the result of Theorem \ref{tau}.\\
(iv) It is an interesting example exhibiting that the Hausdorff dimension of the intersection of two fractal sets with Hausdorff dimension $1/2$ can be any value between $0$ and $1/2$.

\end{remark}

To prove Theorem \ref{tau}, by the formula (\ref{cal}), it is equivalent to consider the Hausdorff dimension of the set
\begin{equation}\label{E}
E(\alpha):=\left\{x\in\Lambda: \liminf\limits_{n\to\infty}\frac{\log a_n(x)}{\log n}=\alpha\right\}~(\alpha \geq 0).
\end{equation}

\begin{theorem}\label{cfx}
For any $\alpha\geq0$, we have
\begin{equation*}
\dim_{\rm H}E(\alpha)=
\begin{cases}
0, &\ \ \ \ 0\leq \alpha<1;\cr
(\alpha-1)/(2\alpha), & \ \ \ \ \ \ \alpha\geq 1.
\end{cases}
\end{equation*}
\end{theorem}

\begin{remark}\label{xd}
For any $\alpha \geq 0$, let
\[
F(\alpha):=\left\{x\in[0,1): \liminf\limits_{n\to\infty}\frac{\log a_n(x)}{\log n}=\alpha\right\}.
\]
Then we claim that
\begin{equation}\label{fa}
\dim_{\rm H}F(\alpha)=
\begin{cases}
1\ \ \ \ \ \text{if}\ \alpha=0;\cr
\frac{1}{2}\ \ \ \ \ \text{if}\ \alpha>0.
\end{cases}
\end{equation}

%Then we have \[E(\alpha) = \Lambda\cap F(\alpha).\]

In fact, on one hand, on account of Good \cite[Theorem 1]{lesGood41} and Fan et al.\,\cite[Lemma 3.2]{FLWW09}, we deduce that $F(\alpha)$ has Hausdorff dimension $1/2$ for any $\alpha >0$. On the other hand, since the set of real numbers whose partial quotients are bounded has full Hausdorff dimension (see \cite{Jar28}), we also have $\dim_{\rm H}F(0)=1$.\\
\indent Notice that $E(\alpha)=\Lambda\cap F(\alpha)$. It is natural to try to relate the Hausdorff dimension of this intersection to that of original sets. In general, we expect the Hausdorff dimension of the intersection of two sets to behave well. That is to say, let $E$ and $F$ be  two subsets of $[0,1)$, then
\[
\dim_{\rm H}(E\cap F)=\dim_{\rm H}E+\dim_{\rm H}F-1 \ \ \text{or}\ \
\dim_{\rm H}(E\cap F)=\min\{\dim_{\rm H}E,\ \dim_{\rm H}F\}.
\]
 However, we will show that the behaviour of $E(\alpha)$ does not follow  the above cases. Indeed, it follows from \eqref{fa}, Theorem \ref{cfx} and $\dim_{\rm H} \Lambda =1/2$ that
\begin{equation*}
\dim_{\rm H}E(\alpha) =
\begin{cases}
\dim_{\rm H}(\Lambda\cap F(\alpha))<\dim_{\rm H}\Lambda+\dim_{\rm H}F(\alpha)-1,\ \ \ \text{if}\ \ \ \alpha=0,\cr
\dim_{\rm H}(\Lambda\cap F(\alpha))=\dim_{\rm H}\Lambda+\dim_{\rm H}F(\alpha)-1,\ \ \ \text{if}\ \ \  0<\alpha\leq1;\cr
\dim_{\rm H}(\Lambda\cap F(\alpha))>\dim_{\rm H}\Lambda+\dim_{\rm H}F(\alpha)-1,\ \ \ \text{if}\ \ \  \alpha>1.
\end{cases}
\end{equation*}
Moreover, for any $\alpha\geq0$, we always have
\[\dim_{\rm H}E(\alpha)<\min\{\dim_{\rm H}\Lambda,\ \dim_{\rm H}F(\alpha)\}.\]
There are more different phenomena on the intersection of fractal sets, for instance see Amou and Bugeaud \cite[Section 7]{AB10},
Bugeaud and Durand \cite{BD}, Fan et al.\,\cite[Remark 2]{FLWW13}, Falconer \cite[Chapter 8]{lesFal90}, F\"{a}rm and Persson \cite[Section 4.2]{FP11},
Hawkes \cite{lesHaw} and Levesley, Salp and Velani \cite{lesLSV}.
\end{remark}

To complete the result of Theorem \ref{cfx}, we consider the Hausdorff dimension of the set
\[
E_\phi:=\left\{x\in \Lambda: \liminf\limits_{n\to\infty}\frac{\log a_n(x)}{\phi(n)}=1\right\},
\]
where $\phi:\mathbb{N} \rightarrow\mathbb{R}_+$ is a function such that $\phi(n) \to \infty$ as $n \to \infty$. If $\phi$ satisfies $\limsup_{n \to \infty} \phi(n)/\log n =0$, then we conclude from Theorem \ref{cfx} that the Hausdorff dimension of $E_\phi$ is always equal to zero. When $\lim_{n\to\infty}\phi(n) /\log n=\alpha$, the Hausdorff dimension of $E_\phi$ is same as that of $E(\alpha)$. For the other case, we have

\begin{theorem}\label{ybphi}
Let $\phi:\mathbb{N} \rightarrow\mathbb{R}_+$ be a non-decreasing function and $\phi(n)/\log n \to\infty$ as $n \to \infty$.
Then
\[
\dim_{\rm H}E_\phi=\frac{1}{B+1},
 \]
 where $B$ is defined as
 \[
\log B:=\limsup\limits_{n\to\infty}\frac{\log\phi(n)}{n}.
 \]
\end{theorem}

As an application of Theorem \ref{ybphi}, if $\phi$ is a power function, then $E_\phi$ is of Hausdorff dimension 1/2; if $\phi$ is an exponential function of the form $\phi(n)=\alpha b^n$ with $\alpha >0$ and $b>1$,
then $E_\phi$ has Hausdorff dimension $1/(b+1)$. The following figure provides a full illustration of the Hausdorff dimension of $E_\phi$, which deeply reveals how $\dim_{\rm H}E_\phi$ changes along with different functions $\phi$.
\begin{figure}[H]\label{Ephi}
\centering
\includegraphics[width=13.5cm,height=7.0cm]{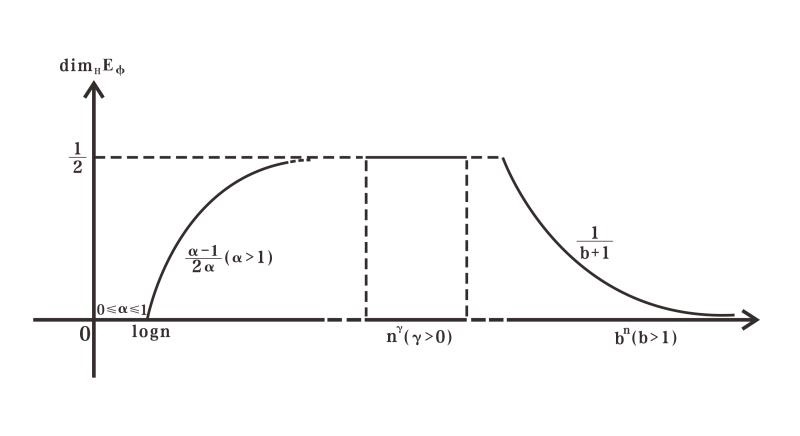}\\
\small{Figure 1: {$\dim_{\rm H} {E}_\phi$} over different functions}
\end{figure}

\section{Preliminaries}
In this section, we recall some definitions and basic properties of continued fractions, see \cite{lesIK02}. Some useful lemmas for calculating the Hausdorff dimension of fractal sets are also given.
\subsection{Elementary properties of continued fractions}
Let $x\in[0,1)\setminus\mathbb{Q}$ and its continued fraction expansion $x=[a_1(x),a_2(x),\cdots,a_n(x),\cdots]$. For any $n\geq1$, we denote by
\[\frac{p_n(x)}{q_n(x)}:=[a_1(x),a_2(x),\cdots,a_n(x)],
\]
which is called the $n$-th convergent of the continued fraction expansion of $x$, where $p_n(x)$ and $q_n(x)$ are co-prime. It is well known that $p_n(x)$ and $q_n(x)$ satisfy the following recursive formula:
\[p_{-1}=1,\ \ p_0=0,\ \ p_n=a_np_{n-1}+p_{n-2}\ (n\geq2),\]
\[q_{-1}=0,\ \ q_0=1,\ \ q_n=a_nq_{n-1}+q_{n-2}\ (n\geq2).\]
Hence $q_n\geq q_{n-1}+q_{n-2}$, which implies that
\begin{equation}\label{fn}
q_n\geq\frac{1}{\sqrt{5}}\left(\frac{1+\sqrt{5}}{2}\right)^{n}-
\frac{1}{\sqrt{5}}\left(\frac{1-\sqrt{5}}{2}\right)^{n}\geq\frac{1}{2\sqrt{5}}\left(\frac{1+\sqrt{5}}{2}\right)^{n}.
\end{equation}
\begin{definition}\label{cylinder}
For any $n\geq1$ and $a_1,a_2,\cdots,a_n\in\mathbb{N}$, we call
\[I(a_1,a_2,\cdots,a_n): =\left\{x\in[0,1]:\ a_1(x)=a_1,a_2(x)=a_2,\cdots,a_n(x)=a_n\right\}\]
the cylinder of order $n$ of continued fractions.
\end{definition}

\begin{proposition}{\cite[p. 18]{lesIK02}}\label{cd}
For any $a_1, a_2,\cdots, a_n\in\mathbb{N}$, the $n$-th order cylinder $I(a_1, a_2,\cdots, a_n)$ is an interval with the endpoints
\[\frac{p_n}{q_n}\ \ \  \text{and}\ \ \ \frac{p_n+p_{n-1}}{q_n+q_{n-1}}.\]
As a result, the length of $I(a_1, a_2,\cdots, a_n)$ equals to
\begin{equation}\label{cdgs}
|I(a_1, a_2,\cdots, a_n)|=\frac{1}{q_n(q_n+q_{n-1})}.
\end{equation}
\end{proposition}

\subsection{Some useful lemmas to estimate Hausdorff dimension }

We collect and establish some useful lemmas for estimating the Hausdorff dimension of certain sets in continued fractions.
The first lemma is due to Fan at al.\,\cite{FLWW09}, which has been used by Fan at al.\,\cite{FLWW13}, Liao and Rams \cite{LR16, lesLR16S} to calculate Hausdorff dimension of fractal sets related to the growth rate of partial quotients.

\begin{lemma}\cite[Lemma 3.2]{FLWW09}\label{flww}
Let $\{s_n\}_{n\geq1}$ be a sequence of positive integers tending to infinity with $s_n\geq3$ for all $n\geq1$. Then for any positive number $N\geq2$,
\[\dim_{\rm H}\{x\in[0,1): s_n\leq a_n(x)<Ns_n, \forall n\geq1\}=
\liminf\limits_{n\to\infty}\frac{\log(s_1s_2\cdots s_n)}{2\log(s_1s_2\cdots  s_n)+\log s_{n+1}}.\]
\end{lemma}

The points in the set $\{x\in[0,1): s_n\leq a_n(x)<Ns_n, \forall n\geq1\}$ cannot guarantee that
their partial quotients are non-decreasing, so we cannot apply the result of Lemma \ref{flww} directly for establishing the lower bound for $\dim_{\rm H} E(\alpha)$.
This leads us to modify the set in Lemma \ref{flww} and calculate its Hausdorff dimension.
To this end, we introduce a useful lemma (see \cite[Example 4.6]{lesFal90}), which provides a method to get a lower bound for the Hausdorff dimension of a fractal set.

\begin{lemma}\label{cxj}
Let $\mathbb{E}=\bigcap_{n\geq0}\mathbb{E}_n$, where $[0,1]=\mathbb{E}_0\supset \mathbb{E}_1\supset\cdots$ is a decreasing sequence of subsets in $[0,1]$ and $\mathbb{E}_n$ is a union of
a finite number of disjoint closed intervals such that each interval in $\mathbb{E}_{n-1}$ contains at least $m_n$ intervals in $\mathbb{E}_n$ which are separated by
gaps of lenghs at least $\varepsilon_n$. If $m_n\geq2$ and $\varepsilon_n>\varepsilon_{n+1}>0$, then
\[\dim_{\rm H} \mathbb{E}\geq\liminf\limits_{n\rightarrow\infty}\frac{\log(m_1m_2\cdots m_n)}{-\log(m_{n+1}\varepsilon_{n+1})}.\]
\end{lemma}

As a consequence of Lemma \ref{cxj}, we have the following result which plays an important role in the proofs of our main results. For completeness, we give a direct and simple proof. See Liao and Rams \cite{lesLR19} for more general result in $d$-decaying Gauss like iterated function systems.

%\begin{lemma}[\cite{FLWW09}]\label{flww}
%Let $\{s_n\}_{n\geq1}$ be a sequence of positive integers tending to infinity with $s_n\geq3$ for all $n\geq1$. For any positive number $N\geq2$,
%\[\dim_{\rm H}\left\{x\in[0,1): s_n\leq a_n(x)<Ns_n, \forall n\geq1\right\}=
%\liminf\limits_{n\to\infty}\frac{\log(s_1s_2\cdots s_n)}{2\log(s_1s_2\cdots  s_n)+\log s_{n+1}}.\]
%\end{lemma}
%Based on Lemma \ref{flww}, we construct kinds of Cantor type sets and give the formula for their exact dimensions.
%

\begin{lemma}\label{geshu}
Let $\{s_n\}_{n\geq1}$ be a sequence of positive integers tending to infinity with $s_n\geq3$ for all $n\geq1$. Write
\[
\mathbb{F}=\big\{x\in[0,1): ns_n\leq a_n(x)<(n+1)s_n, \forall\ n\geq1\big\}.
\]
Then
\[\dim_{\rm H}\mathbb{F}=\frac{1}{2+\xi},\ \ \text{with}\
\xi:=\limsup\limits_{n\to\infty}\frac{2\log(n+1)!+
\log s_{n+1}}{\log(s_1 s_2\cdots s_n)}.\]
\end{lemma}

\begin{proof}
\textbf{Lower bound:} We will use Lemma \ref{cxj} to get the lower bound of $\dim_{\rm H}\mathbb{F}$. In order to describe the construction of $\mathbb{E}_n$, we shall make use of the symbolic space, which is explained as follows: for any $n\geq1$, we define
    \[
    \mathcal{D}_n=\big\{(\sigma_1,\cdots,\sigma_n)\in\mathbb{N}^{n}:\ ks_k\leq \sigma_k<(k+1)s_k,\forall\ 1\leq k\leq n\big\}.
    \]
    For any $n\geq1$ and $(\sigma_1,\cdots,\sigma_n)\in \mathcal{D}_n$, we denote
    \[ J(\sigma_1,\cdots,\sigma_n)=
    cl\bigcup_{\sigma_{n+1}} I(\sigma_1,\cdots,\sigma_n,\sigma_{n+1})\]
    and call it the basic interval of order $n$, where the union is take over all $\sigma_{n+1}$ such that $(\sigma_1,\cdots,\sigma_n,\sigma_{n+1})\in \mathcal{D}_{n+1}$. % $I(\sigma_1,\cdots,\sigma_n,\sigma_{n+1})$ is the $(n+1)$-th cylinder for continued fractions.
    Put $\mathbb{E}_0 \equiv [0,1]$ and
    \[
    \mathbb{E}_n=\bigcup\limits_{(\sigma_1,\cdots,\sigma_n)\in \mathcal{D}_n}J(\sigma_1,\cdots,\sigma_n)\ \ \text{for all}\ n\geq1.
    \]
    Then
    \[\mathbb{E}=\bigcap_{n=0}\mathbb{E}_n=\bigcap_{n=0}\bigcup\limits_{(\sigma_1,\cdots,\sigma_n)\in \mathcal{D}_n}J(\sigma_1,\cdots,\sigma_n)=\mathbb{F}.\]
    It follows from the construction of $\mathbb{E}$ that each element in $\mathbb{E}_{n-1}$ contains some numbers of the basic intervals of order $n$ in $\mathbb{E}_n$, we denote this number by $m_n$. For all $n\geq1$,
    \[
    m_n=(n+1)s_n-1-ns_n=s_n-1\geq2.
    \]
Next we are going to estimate the gaps between two basic intervals of the same order. For any $n\geq1$, let $J(\tau_1,\tau_2,\cdots,\tau_n)$ and $J(\sigma_1,\sigma_2,\cdots,\sigma_n)$ be the two distinct basic intervals in $\mathbb{E}_n$. We know that the intervals $J(\tau_1,\tau_2,\cdots,\tau_n)$ and $J(\sigma_1,\sigma_2,\cdots,\sigma_n)$ are separated by the $(n+1)$-th order cylinder
 \[
 I(\tau_1,\tau_2,\cdots,\tau_n,1) \ \ \ \text{or}\ \ \ I(\sigma_1,\sigma_2,\cdots,\sigma_n,1)
 \]
according to the location between $J(\tau_1,\tau_2,\cdots,\tau_n)$ and $J(\sigma_1,\sigma_2,\cdots,\sigma_n)$. By \eqref{cdgs},
\begin{align*}
    |I(\sigma_1,\sigma_2,\cdots,\sigma_n,1)|&\geq \frac{1}{2q^2_{n+1}}\geq\frac{1}{8}\left(\prod\limits_{k=1}^{n}(\sigma_k+1)\right)^{-2}\\
    &\geq \frac{1}{8}\left(\prod\limits_{k=1}^{n}\big((k+1)s_k\big)\right)^{-2}:=\varepsilon_n.
\end{align*}
It is easy to see that for all $n\geq1$, $0<\varepsilon_{n+1}<\varepsilon_n$. Similarly, we can also obtain $|I(\tau_1,\cdots,\tau_k,1)|\geq\varepsilon_n$.
By Lemma \ref{cxj}, we have
 \begin{align*}
    \dim_{\rm H}\mathbb{F} &\geq \liminf\limits_{n\to\infty}\frac{\sum\limits_{k=1}^{n}\log (s_k-1)}
    {-\log(s_{n+1}-1)+\log8+2\sum\limits_{k=1}^{n+1}\log\big((k+1)s_k\big)}\\
    &= \frac{1}{2+\limsup\limits_{n\to\infty}\frac{2\log(n+1)!+
\log s_{n+1}}{\log(s_1 s_2\cdots s_n)}}.
    \end{align*}

\textbf{Upper bound:} To obtain the upper bound of $\dim_{\rm H}\mathbb{F}$, we need to find a cover of the set $\mathbb{F}$.
Let
\[ J(a_1,a_2,\cdots,a_n)=
   cl \bigcup_{a_{n+1}\geq(n+1)s_{n+1}} I(a_1,a_2,\cdots,a_n,a_{n+1}).\]
Then we have
    \[\mathbb{F}\subseteq\bigcap_{n=1}^{\infty}\bigcup_{ks_k\leq a_k <(k+1)s_k,1\leq k\leq n}J(a_1,a_2,\cdots,a_n).\]
It means that for any $n \geq 1$,
\[
\bigcup_{ks_k\leq a_k <(k+1)s_k,1\leq k\leq n}J(a_1,a_2,\cdots,a_n)
\]
is a cover of $\mathbb{F}$. The number of such basic intervals is
\[\prod\limits_{k=1}^{n}\big((k+1)s_k-1-ks_k\big)=\prod\limits_{k=1}^{n}\big(s_k-1\big),\]
and the length of such a basic interval is estimated as
\begin{align*}
|J(a_1,a_2,\cdots,a_n)|&\geq \frac{1}{3}
|I(a_1,a_2,\cdots,a_n)|\times\sum\limits_{a_{n+1}(x)\geq(n+1)s_{n+1}}\frac{1}{a^{2}_{n+1}(x)}\\
&\geq \frac{1}{6q^2_n}\times\frac{1}{(n+1)s_{n+1}}\\%=\frac{1}{6M(n+1)s_{n+1}q^2_n}
&\geq \frac{1}{6(n+1)s_{n+1}\cdot2^n\cdot\big(\prod\limits_{k=1}^{n}((k+1)s_k)\big)^2}.
\end{align*}
Therefore, we obtain the upper bound
\begin{align*}
\dim_{\rm H}\mathbb{F}&\leq \liminf\limits_{n\to\infty}
\frac{\sum\limits_{k=1}^{n}\log(s_k-1)}{-\log|J(a_1,a_2,\cdots,a_n)|}\\
&\leq \liminf\limits_{n\to\infty}\frac{\sum\limits_{k=1}^{n}\log(s_k-1)}
{\log\big(6(n+1)\big)+\log s_{n+1}+2\sum\limits_{k=1}^{n}\log s_k+2\log(n+1)!}\\
&= \frac{1}{2+\limsup\limits_{n\to\infty}\frac{2\log(n+1)!+
\log s_{n+1}}{\log(s_1 s_2\cdots s_n)}}.
\end{align*}
\end{proof}

\indent The following result essentially comes from Good \cite{lesGood41} and {\L}uczak \cite{lesLuc97}, which provides an important tool to estimate the upper bound for the Hausdorff dimension of some fractal sets. Such a result also appeared in the PhD thesis of Baowei Wang in Chinese,
we here write its proof for completeness.

\begin{lemma}\label{wlw}
Let $\varphi: \mathbb{N} \to \mathbb{R}_+$ be a function such that $\varphi(n) \to \infty$ as $n \to \infty$. Then
\[
\dim_{\rm H}\big\{x\in[0,1): a_n(x)\geq \varphi(n) , \forall\  n\geq 1\big\}=\frac{1}{B+1},
\]
where $B$ is given by
\[\log B:=\limsup\limits_{n\to\infty}\frac{\log\log\varphi(n)}{n}.\]
\end{lemma}

\begin{proof}
Let $a,b >1$ be fixed. It follows from Good \cite[Lemma 1]{lesGood41} and the main theorem of {\L}uczak \cite{lesLuc97} that for any $N_0 \geq 1$,
\begin{align}\label{Luc}
\dim_{\rm H}\big\{x\in[0,1): a_n(x)&\geq a^{b^n}, \forall\  n\geq N_0 \big\} \notag\\
&= \dim_{\rm H}\big\{x\in[0,1): a_n(x)\geq a^{b^n} i.o.~n\big\} =\frac{1}{b+1},
\end{align}
where $i.o.$ means infinitely often.

{\bf (i) $B=1$.} For any $\varepsilon>0$, by the definition of $B$, there exists $N^\prime >0$ such that $\varphi(n) \leq e^{(1+\varepsilon)^n}$ holds for all\  $n \geq N^\prime$. This implies that
\[\big\{x\in[0,1): a_n(x) \geq e^{(1+\varepsilon)^n}, \forall\  n\geq N^\prime \big\}\subseteq\big\{x\in[0,1): a_n(x)\geq \varphi(n) , \forall\  n\geq N^\prime  \big\}.\] By (\ref{Luc}), we have
\[\dim_{\rm H}\big\{x\in[0,1): a_n(x)\geq \varphi(n) , \forall\  n\geq N^\prime  \big\} \geq  1/(2+\varepsilon),\]
letting $\varepsilon \to 0^+$,
\[\dim_{\rm H}\big\{x\in[0,1): a_n(x)\geq \varphi(n) , \forall\  n\geq N^\prime  \big\} \geq 1/2.\]
 From Lemma 1 of Good \cite{lesGood41}, we have the lower bound
\[
\dim_{\rm H}\big\{x\in[0,1): a_n(x)\geq \varphi(n) , \forall\  n\geq 1\big\} \geq \frac{1}{2}.
\]
Since $\varphi(n) \to \infty$ as $n \to \infty$, we have
 \[\big\{x\in[0,1): a_n(x)\geq \varphi(n) , \forall\ n\geq 1\big\}\subseteq\big\{x\in[0,1): a_n(x) \to \infty\  \text{as}\ n \to \infty\big\}.\]
 %The latter is of Hausdorff dimension $1/2$ and thus the upper bound is obtained.
Notice that
\[\dim_{\rm H}\big\{x\in[0,1): a_n(x) \to \infty\  \text{as}\ n \to \infty\big\}=1/2.\]
Thus the upper bound is obtained.\\
{\bf (ii) $B \in (1,\infty)$.} Analysis similar to that in the part {\bf(i)}, we conclude that
\[
\dim_{\rm H}\big\{x\in[0,1): a_n(x)\geq \varphi(n) , \forall\  n\geq 1\big\} \geq \frac{1}{B+1}.
\]
For the upper bound, by the definition of $B$, for any $\varepsilon>0$, we have $\varphi(n) \geq e^{(B+\varepsilon)^n}$ holds for infinitely many $n \in \mathbb{N}$. This means that
\[\big\{x\in[0,1): a_n(x)\geq \varphi(n) , \forall\  n\geq 1\big\}\subseteq\big\{x\in[0,1): a_n(x)\geq e^{(B+\varepsilon)^n} i.o.~n\big\}.\]
Notice that
%the latter is of Hausdorff dimension is $1/(B+1+\varepsilon)$.
\[\dim_{\rm H}\big\{x\in[0,1): a_n(x)\geq e^{(B+\varepsilon)^n} i.o.~n\big\}=1/(B+1+\varepsilon).\]
 Let $\varepsilon \to 0^+$, the upper bound follows.

{\bf (iii) $B=\infty$.} In this case, for any large $M>2$, we know $\varphi(n) \geq e^{M^n}$ holds for infinitely many $n \in \mathbb{N}$. Then we see that
\[
\dim_{\rm H}\big\{x\in[0,1): a_n(x)\geq \varphi(n) , \forall\ n\geq 1\big\} \leq \frac{1}{M+1}.
\]
The result follows by letting $M \to \infty$.
\end{proof}
\begin{remark}
For a little more general version of this lemma, we have
\[
\dim_{\rm H}\big\{x\in[0,1): a_n(x)\geq \varphi(n) , \forall\  n\geq N \big\}=\frac{1}{B+1}
\]
holds for all $N \geq 1$.
\end{remark}
\section{Proofs of main results}
In this section, we will prove the main results of this paper.

\subsection{The proof of Theorem \ref{Fu}}
(i) To prove $\tau: [0,1) \to [0,\infty]$ is measurable, it is sufficient to show that for any $\alpha \in (0,\infty)$, the set $\{x \in [0,1): \tau(x)<\alpha\}$ is a Borel set.
Now let $\alpha \in (0,\infty)$ be fixed. It is easy to verify that
\begin{align}\label{Mk}
\big\{x \in [0,1): \tau(x)<\alpha\big\} &= \bigcup_{k\geq k_0}\left\{x \in [0,1): \sum_{j\geq 1} a^{-(\alpha -1/k)}_j(x)<\infty \right\} \nonumber\\
&=\bigcup_{k\geq k_0} \bigcap_{\ell\geq 1} \bigcup_{m\geq 1} \bigcap_{n\geq 1} M^\alpha_{k,\ell,m,n}
\end{align}
where $k_0:= \lfloor1/\alpha\rfloor+1$ and the set $M^\alpha_{k,\ell,m,n}$ is given by
\[
 M^\alpha_{k,\ell,m,n}= \left\{x\in [0,1):  \sum^{m+n}_{j=m+1} a^{-(\alpha -1/k)}_j(x)\leq \frac{1}{\ell}\right\}.
\]
For each $(k,\ell,m,n)$, we see that the set $M^\alpha_{k,\ell,m,n}$ is a countable union of cylinders of finite order, which are intervals and of the form $[a,b)$ or $(c,d]$ with rational numbers $a,b,c,d$ in $[0,1)$.
Hence $M^\alpha_{k,\ell,m,n}$ is a Borel set. It follows from (\ref{Mk}) that the set $\{x \in [0,1): \tau(x)<\alpha\}$ is a Borel set.

(ii) The result is a direct consequence of Birkhoff's ergodic theorem. Indeed, we know that the continued fraction dynamical system $([0,1), T_{G}, \mu)$ is ergodic, where $T_{G}$ is the Gauss map and $\mu$ is the Gauss measure which is equivalent to the Lebesgue measure (see \cite{lesIK02} for their definitions). For any large $t>0$, the classical Birkhoff ergodic theorem shows that for $\mu$-almost all $x \in [0,1)$,
\begin{equation}\label{ergodic}
\lim_{n \to \infty} \frac{a^{-t}_1(x)+\cdots+a^{-t}_n(x)}{n} = \int^1_0 a^{-t}_1(x) d\mu(x):=P(t).
\end{equation}
By the relation between Gauss measure and Lebesgue measure, we obtain the above result also holds for Lebesgue almost all $x \in [0,1)$ and $P(t)$ satisfies
\[
\frac{1}{2\log 2}\sum_{k \geq 1} \frac{1}{k^{t+1}(k+1)}\leq P(t) \leq \frac{1}{\log 2} \sum_{k \geq 1} \frac{1}{k^{t+1}(k+1)}.
\]
Combining this with (\ref{ergodic}) and the definition of $\tau$, we have $\tau(x) = \infty$ for Lebesgue almost all $x \in [0,1)$.

(iii) We divide $\alpha\in[0,\infty]$ into three cases. For $\alpha =0$, let $x:=[a_1 ,a_2 ,\cdots,a_n ,\cdots]$ with $a_n =\lfloor e^n\rfloor$ for all $n \geq 1$. Then $x\in [0,1)$ is an irrational number and $\tau(x) =0$.
For $\alpha \in (0,\infty)$, let $x:=[a_1 ,a_2 ,\cdots,a_n ,\cdots]$ with $a_n =\lfloor n^{1/\alpha}\rfloor$ for all $n \geq 1$. Thus we have $x\in [0,1)$ is an irrational number and $\tau(x) =\alpha$.
For $\alpha =\infty$, let $x$ be the golden ratio. Of course it is an irrational and its partial quotients $a_n(x) =1$ for all $n \geq 1$ and hence $\tau(x) =\infty$.

(iv)\ By (iii), there exists an irrational number $x_0 \in [0,1)$ such that $\tau(x_0) =\alpha$.
Let $(\varepsilon_1,\varepsilon_2,\cdots,\varepsilon_n,\cdots) \in \{0,1\}^{\mathbb{N}}$. We construct a new real number $\widetilde{x}_0$ as $[\widetilde{a}_1,\widetilde{a}_2,\cdots,\widetilde{a}_n,\cdots]$ with $\widetilde{a}_k = a_k(x_0)+\varepsilon_k$ for all $k \geq 1$. Then we have $\widetilde{x}_0 \in [0,1)$ and $\tau(\widetilde{x}_0) =\alpha$. This implies that there are uncountably infinite many $x \in [0,1)$ such that $\tau(x)=\alpha.$
In the following, we will see that the level set $\{x \in [0,1): \tau(x) =\alpha\}$ is dense in $[0,1)$.
To do this, write
\[
M(x_0) = \bigcup_{N \geq 1}\big\{x\in [0,1): a_n(x)= a_n(x_0), \forall\  n\geq N\big\}.
\]
For any $y \in M(x_0)$ and $\sigma >0$, we have the series $\sum_{n \geq 1}a^{-\sigma}_n(y)$ converges if and only if the series $\sum_{n \geq 1}a^{-\sigma}_n(x_0)$ converges.
This yields that $\tau(y) =\alpha$, i.e.,
\[M(x_0)\subseteq\{x \in [0,1): \tau(x) =\alpha\}.\]
Next we will prove that $M(x_0)$ is dense in $[0,1)$. In fact, for any $y \in [0,1)$,
if $y$ is rational, let $y = [a_1(y),a_2(y),\cdots,a_k(y)]$ for some $k \geq 1$ and
\[
x_n = [a_1(y),a_2(y),\cdots,a_k(y),n,a_{k+2}(x_0), a_{k+3}(x_0),\cdots],\ \ \forall\  n \geq 1;
\]
if $y$ is irrational, let $y = [a_1(y),a_2(y),\cdots,a_n(y),\cdots]$ and
\[
x_n = [a_1(y),a_2(y),\cdots,a_n(y),a_{n+1}(x_0), a_{n+2}(x_0),\cdots],\ \ \forall\  n \geq 1.
\]
In both cases, we obtain $x_n \in M(x_0)$ for all $n \geq 1$ and $x_n \to y$ as $n \to \infty$. Hence the level set $\{x \in [0,1): \tau(x) =\alpha\}$ is dense in $[0,1)$.

(v) For any interval $I\subseteq [0,1)$, it is trivial that $\tau(I) \subseteq [0,\infty]$. For any $\alpha \in [0,\infty]$, by (iv), we know that the set $\{x \in [0,1): \tau(x) =\alpha\}$ is dense in $[0,1)$.
Thus we obtain the intersection of $I$ and $\{x \in [0,1): \tau(x) =\alpha\}$ is not empty. Hence we can find an $x_0 \in I$ such that $\tau(x_0)=\alpha$, that is, $[0,\infty] \subseteq \tau(I)$.

(vi) Let $y \in [0,1)$ be fixed and let $\alpha_0:= \tau(y)$. For an $\alpha \neq \alpha_0$, by (iv), the level set $\{x \in [0,1): \tau(x) =\alpha\}$ is dense in $[0,1)$.
For the given $y$, there exists $\{x_n\}_{n \geq 1}$ such that $\tau(x_n) = \alpha$ for all $n \geq 1$ and $x_n \to y$ as $n \to \infty$. Note that $\tau(y) = \alpha_0 \neq \alpha = \tau(x_n)$, we see that the function $\tau$ is discontinuous at $y$. Therefore, the function $\tau: [0,1) \to [0,\infty]$ is everywhere discontinuous.

\subsection{The proof of Theorem \ref{No}}
We divide the proof into two cases: $\alpha =\infty$ and $\alpha \in [0,\infty)$.

In the case when $\alpha =\infty$, by Theorem \ref{Fu} (ii), we have $\tau(x) = \infty$ for Lebesgue almost all $x \in [0,1)$. It is clear to see that the set $\{x \in [0,1): \tau(x) =\infty\}$ has full Hausdorff dimension.

In the following, we assume that $0 \leq \alpha <\infty$. If $\tau(x_0) =\alpha$ for some $x_0 \in [0,1)$, by the definition of $\tau$, we have there exists $\varepsilon_0>0$ such that
 \[\sum_{n \geq 1} a^{-(\alpha+\varepsilon_0)}_n(x_0) <\infty, \] which implies $a_n(x_0) \to \infty$ as $n \to \infty$.
It follows from Good \cite[Theorem 1]{lesGood41} that the set of real numbers $x \in [0,1)$ such that $a_n(x) \to \infty$ as $n \to \infty$ has Hausdorff dimension $1/2$.
Thus we obtain
\[
\dim_{\rm H} \{x \in [0,1): \tau(x) =\alpha \} \leq1/2.
\]
The lower bound of the set $\{x \in [0,1): \tau(x) =\alpha\}$ follows from Lemma 3.2 of \cite{FLWW09} (see Lemmas \ref{flww}). More precisely, we have

\begin{enumerate}
\item for $\alpha =0$, $\{x \in [0,1): e^n \leq a_n(x) < 2e^n,\forall\  n \geq 1\}\subseteq\{x \in [0,1): \tau(x) =\alpha\}$
and the former is of Hausdorff dimension $1/2$;
\item  for $\alpha \in (0,\infty)$,
$
\{x \in [0,1): n^{1/\alpha}\leq a_n(x) <2n^{1/\alpha}, \forall\  n \geq 1\}\subseteq\{x \in [0,1): \tau(x) =\alpha\}
$
and also the former is of Hausdorff dimension $1/2$.
\end{enumerate}
Therefore, for any $\alpha\in[0,\infty)$, $\dim_{\rm H} \{x \in [0,1): \tau(x) =\alpha\}\geq1/2$.

% for the case$\alpha=0$ and $\alpha \in (0,\infty)$ respectively

\subsection{The proof of Theorem \ref{cfx}}
Firstly, we give a useful combinatorial lemma concerning about the cardinality of sets of numbers.
\begin{lemma}\label{card}
For any positive integer $L$ and $n \geq 1$, let
\[
A=\big\{(\sigma_1,\sigma_2,\cdots,\sigma_n)\in\mathbb{N}^{n}: 1\leq\sigma_1\leq\sigma_2\leq\cdots \leq\sigma_n\leq L\big\}
\]
and $N_{n}(L)=\# A$. Then
 \[N_{n}(L)=\binom{n}{L+n-1}=\frac{(n+L-1)!}{n!(L-1)!}.\]
\end{lemma}

\begin{proof}
Write
\[A^{'}=\{(r_1, r_2,\cdots, r_n)\in\mathbb{N}^{n}: 1\leq r_1<r_2<\cdots<r_n\leq L+n-1\}.\]
We consider a map $f:\ A\rightarrow A^{'}$ defined as follows: for any $(\sigma_1,\sigma_2,\cdots,\sigma_n)\in A$,
\[f(\sigma_1,\sigma_2,\cdots,\sigma_n)=(r_1, r_2,\cdots, r_n)=(\sigma_1,\sigma_2+1,\cdots,\sigma_n+n-1).\]
It is easy to check that $f$ is bijection. Thus the sets $A$ and $A^{'}$ have the same cardinality. The number of elements in $A^{'}$ is equal to a combinatorial number, i.e, the number of $n$ numbers can be chosen from among $L+n-1$ numbers. That is,
\[N_{n}(L)=\sharp A^{'}=\binom{n}{L+n-1}=\frac{(n+L-1)!}{n!(L-1)!}.\]
\end{proof}
Recall that
\[
E(\alpha)=\left\{x\in \Lambda:
 \liminf\limits_{n\to\infty}\frac{\log a_n(x)}{\log n}=\alpha\right\}.
 \]

\subsubsection{The case $\alpha\geq1$}
\textbf{Upper bound:} For any $0<\varepsilon<\alpha$, we have
\[E(\alpha)\subseteq \bigcup_{N=1}B_{N}(\varepsilon),\]
where
\begin{equation}\label{bhgx}
 B_{N}(\varepsilon)=\bigcap_{n=N}\bigcup_{k=n}\big\{x\in\Lambda: a_k(x)\leq k^{\alpha+\varepsilon}, a_{j}(x)\geq j^{\alpha-\varepsilon},\ \forall\ N\leq j\leq k\big\}.
 \end{equation}
 Then we obtain
 \begin{equation}\label{wsgx1}
 \dim_{\rm H}E(\alpha)\leq\sup_{N\geq1}\{\dim_{\rm H}B_{N}(\varepsilon)\}.
 \end{equation}
 For similarity, we only deal with the Hausdorff dimension of $B_{1}(\varepsilon)$ since the proof for other case is similar.
%We claim that
%\[\dim_{\rm H}B_{N}(\varepsilon) \leq (\alpha+2\varepsilon-1)/2(\alpha-\varepsilon)\ \text{ for all }\ N %\geq 1.\]
%If so, by (\ref{wsgx1}), letting $\varepsilon \to 0^+$,  we have
 %\[\dim_{\rm H}E(\alpha)\leq\alpha -1)/(2\alpha).\] Now it remains to prove the claim.
In view of (\ref{bhgx}), for any $n \geq 1$, we have
\begin{eqnarray}\label{11}
\nonumber B_{1}(\varepsilon)&\subseteq&\bigcup_{k=n}\big\{x\in\Lambda: a_k(x)\leq k^{\alpha+\varepsilon}, a_{j}(x)\geq j^{\alpha-\varepsilon},\ \forall\  1\leq j\leq k\big\}\\
&\subseteq&\bigcup_{k=n}\bigcup_{(\sigma_1, \cdots, \sigma_k) \in \mathcal{C}_k}I(\sigma_1, \cdots, \sigma_k)
\end{eqnarray}

 %Furthermore,
% \[\{x\in\Lambda: a_k(x)\leq k^{\alpha+\varepsilon}, a_{j}(x)\geq j^{\alpha-\varepsilon},\ \forall\  %1\leq j\leq k\}\subseteq
%\bigcup_{(\sigma_1, \cdots, \sigma_k) \in \mathcal{C}_k}I(\sigma_1, \cdots, \sigma_k),\]
where
 \[
 \mathcal{C}_k=\left\{(\sigma_1, \cdots, \sigma_k)\in\mathbb{N}^{k}:1\leq \sigma_1\leq\cdots \leq \sigma_k \leq k^{\alpha+\varepsilon}, \sigma_j\geq j^{\alpha-\varepsilon}, \forall \ 1\leq j\leq k\right\}.
 \]
By \eqref{cdgs}, we get
 \begin{equation}\label{zzss}
 |I(\sigma_1,\cdots, \sigma_k)| \leq\frac{1}{q^{2}_k} \leq\prod\limits_{i=1}^{k}a^{-2}_{i} \leq (k!)^{-2(\alpha-\varepsilon)}.
 \end{equation}
It follows from Lemma \ref{card} that
 \begin{align}\label{deltags}
\#\mathcal{C}_k&\leq N_{k}\left(\lfloor k^{\alpha+\varepsilon}\rfloor\right)\nonumber\\
&= \frac{\lfloor k^{\alpha+\varepsilon}\rfloor\cdot \left(\lfloor k^{\alpha+\varepsilon}\rfloor+1\right)\cdots\left(\lfloor k^{\alpha+\varepsilon}\rfloor+k-1\right)}{k!} \nonumber\\
&= \frac{k^{k(\alpha+\varepsilon)}}{k!}\cdot\left(1+\frac{1}{k^{\alpha+\varepsilon}}\right)\cdots\left(1+\frac{k-1}{k^{\alpha+\varepsilon}}\right) \leq \frac{2^k\cdot k^{k(\alpha+\varepsilon)}}{k!}.
 \end{align}
From the Stirling formula, for all $n \geq 1$,
\begin{equation}\label{st}
\sqrt{2\pi}n^{n+\frac{1}{2}}e^{-n}\leq n!\leq en^{n+\frac{1}{2}}e^{-n}.
\end{equation}
 Combining \eqref{st} with \eqref{deltags}, we deduce that
 \begin{equation}\label{jqgs}
 \#\mathcal{C}_k \leq 2^{k}\cdot (k!)^{\alpha+\varepsilon-1}\cdot\left(\frac{e^k}{\sqrt{2\pi k}}\right)^{\alpha+\varepsilon}.
 \end{equation}
Taking $s= \frac{\alpha+2\varepsilon-1}{2(\alpha-\varepsilon)}$, by the definition of $s$-dimensional Hausdorff measure, we conclude from \eqref{11}, \eqref{zzss} and \eqref{jqgs} that
\begin{align*}\label{qyo}
 \nonumber \mathcal{H}^{s}(B_1(\varepsilon))
 &\leq \liminf_{n\to\infty}\sum\limits_{k=n}^{\infty}\sum\limits_{(\sigma_1, \cdots, \sigma_k) \in \mathcal{C}_k}|I(\sigma_1, \cdots,\sigma_k)|^{s} \\
 &\leq \liminf_{n\to\infty}\sum\limits_{k=n}^{\infty}\frac{\#\mathcal{C}_k }{(k!)^{2s(\alpha-\varepsilon)}} \leq \liminf_{n\to\infty}\sum\limits_{k=n}^{\infty}\frac{2^ke^{k(\alpha+\varepsilon)}}
 {(k!)^{\varepsilon}}=0,
 \end{align*}
Thus
\[\dim_{\rm H} B_{1}(\varepsilon) \leq \frac{\alpha+2\varepsilon-1}{2(\alpha-\varepsilon)}.\]
Letting $\varepsilon \to 0^+$ and then by (\ref{wsgx1}),  we have
\[\dim_{\rm H}E(\alpha)\leq\dim_{\rm H} B_{1}(\varepsilon)\leq\frac{\alpha-1}{2\alpha}.\]
\textbf{Lower bound:} It is obvious that $\dim_{\rm H} E(\alpha) =0$ for $\alpha=1$ according to the proof of its upper bound, so the following we assume that $\alpha>1$. Let $M$ be the positive integer such that  $s_n=M\lfloor n^{\alpha-1}\rfloor\geq3$ for all $n \geq 1$. Write
\[
\mathbb{F}(\alpha)=\big\{x\in [0,1): ns_n\leq a_n(x)<(n+1)s_n, \forall\ n\geq1\big\}.
\]
By Lemma \ref{geshu}, we see that $\dim_{\rm H} \mathbb{F}(\alpha) = 1/(2+\xi)$, where
\[
\xi=\limsup\limits_{n\to\infty}\frac{2\log(n+1)!+\log(M\lfloor (n+1)^{\alpha-1}\rfloor)}{\sum\limits_{i=1}^{n}\log(M\lfloor n^{\alpha-1}\rfloor)} = \frac{2}{\alpha -1}.
\]
Next we claim that
\[\mathbb{F}(\alpha)\subseteq E(\alpha).\] Indeed, on one hand, for any $x \in \mathbb{F}(\alpha)$, it is easy to see that for all $n \geq 1$
\[a_n(x)<(n+1) s_n \leq(n+1)s_{n+1} \leq a_{n+1}(x)\]
 and $a_n(x) \to \infty $ as $n \to \infty$, thus we get $x \in \Lambda$. On the other hand, we also obtain
\begin{align*}
 \nonumber\alpha=\liminf\limits_{n\to\infty}\frac{\log\left(nM \lfloor n^{\alpha-1}\rfloor\right)}{\log n}&\leq \liminf\limits_{n\to\infty}
 \frac{\log a_n(x)}{\log n}\\
 &\leq \liminf\limits_{n\to\infty}\frac{\log\left((n+1)M \lfloor (n+1)^{\alpha-1}\rfloor\right)}{\log n}=\alpha.
\end{align*}
Therefore, we have \[\dim_{\rm H} E(\alpha) \geq \dim_{\rm H} F(\alpha)=(\alpha-1)/(2\alpha).\]

\subsubsection {The case $0\leq\alpha<1$} In this case, we will prove $\dim_{\rm H}E(\alpha) =0$.
 By the definition of $\liminf$, for any $0<\varepsilon<1-\alpha$, we have
 \begin{equation}\label{2bhgx}
E(\alpha)\subseteq\bigcap_{n=1}\bigcup_{k=n}\left\{x\in \Lambda:\ a_k(x)\leq k^{\alpha+\varepsilon}\right\}.
 \end{equation}
 Notice that
 \[\left\{x\in \Lambda: a_k(x)\leq k^{\alpha+\varepsilon}\right\} \subseteq \bigcup_{(\sigma_1, \cdots, \sigma_k)\in \widetilde{\mathcal{C}}_k}I(\sigma_1, \cdots, \sigma_k),\]
 where
 \[
 \widetilde{\mathcal{C}}_k=\left\{(\sigma_1, \cdots, \sigma_k)\in\mathbb{N}^{k}: 1\leq \sigma_1\leq \cdots\leq \sigma_k\leq k^{\alpha+\varepsilon}\right\}.
 \]
 By \eqref{fn} and \eqref{cdgs}, we obtain
 \begin{equation}\label{2cd}
 |I(\sigma_1, \cdots, \sigma_k)| \leq\frac{1}{q^{2}_k}
 \leq\frac{1}{20}\left(\frac{1+\sqrt{5}}{2}\right)^{-2k}.
 \end{equation}
Since $0<\varepsilon<1-\alpha$, it follows from Lemma \ref{card} that
 \begin{align}\label{delta1gs}
\# \widetilde{\mathcal{C}}_k =  N_{k}\left(\lfloor k^{\alpha+\varepsilon}\rfloor\right)
 &\leq (k+1)\cdot(k+2)\cdots(k+\lfloor k^{\alpha+\varepsilon}\rfloor-1) \notag \\
 & \leq \left(k+ k^{\alpha+\varepsilon} \right)^{k^{\alpha+\varepsilon}} \leq e^{(1+\log k)k^{\alpha+\varepsilon}}.
 \end{align}
Taking $s=\varepsilon$, by the definition of $s$-dimensional Hausdorff measure, combining \eqref{2bhgx}, \eqref{2cd} with \eqref{delta1gs}, we deduce that
\begin{align*}
 \mathcal{H}^{s}(E(\alpha))
 &\leq \liminf_{n\to\infty}\sum\limits_{k=n}^{\infty}\sum\limits_{(\sigma_1, \cdots, \sigma_k)\in \widetilde{\mathcal{C}}_k}|I(\sigma_1,\sigma_2,\cdots,\sigma_k)|^{s}\\
 &\leq \liminf_{n\to\infty}\sum\limits_{k=n}^{\infty} \# \widetilde{\mathcal{C}}_k \cdot20^{-\varepsilon}\left(\frac{1+\sqrt{5}}{2}\right)^{-2k\varepsilon}\\
 &\leq  \liminf_{n\to\infty} \sum\limits_{k=n}^{\infty}e^{k^{\alpha+\varepsilon}(1+\log k)}\cdot20^{-\varepsilon}\left(\frac{1+\sqrt{5}}{2}\right)^{-2k\varepsilon}=0,
 \end{align*}
which implies $\dim_{\rm H} E(\alpha)\leq \varepsilon$. Letting $\varepsilon \to 0^+$, the result follows.

\subsection{The proof of Theorem \ref{ybphi}}
The upper bound of $\dim_{\rm H}E_\phi$ is a consequence of Lemma \ref{wlw}. Indeed, for any $0<\varepsilon<1$, we have
\begin{equation*}
E_\phi \subseteq \bigcup_{N=1}^{\infty}\left\{x\in[0,1): a_n(x)\geq e^{(1-\varepsilon)\phi(n)},\ \forall\  n\geq N\right\}.
\end{equation*}
It follows from Lemma \ref{wlw} that for all $N \geq 1$,
\[
\dim_{\rm H}\left\{x\in[0,1): a_n(x)\geq e^{(1-\varepsilon)\phi(n)},\ \forall\  n\geq N\right\} = \frac{1}{B+1},
\]
where $B$ is given by $\log B:=\limsup_{n\to\infty} (\log\phi(n))/n$. By the countable stability of Hausdorff dimension, we obtain
\[
\dim_{\rm H}E_\phi \leq \frac{1}{B+1}.
\]

The proof for the lower bound of $\dim_{\rm H}E_\phi$ is inspired by Liao and Rams \cite{LR16}. It is trivial for $B=\infty$. In the following, we always assume that $1 \leq B <\infty$. Since \[\limsup\limits_{n\to\infty}\frac{\log\phi(n)}{n}=\log B,\]
for any $\varepsilon>0$, we have $\phi(n) \leq (B+\varepsilon/2)^n$ for $n$ large enough. This implies
\[
 \phi(n)(B+\varepsilon)^{j-n} \leq (B+\varepsilon/2)^n(B+\varepsilon)^{j-n} \to 0 \ \ (n \to \infty).
\]
 Let
\begin{equation*}\label{tidy}
 T_j=\sup\limits_{n\geq j}\left\{e^{\phi(n)(B+\varepsilon)^{j-n}}\right\},\ j=1,2,\cdots.
\end{equation*}

Then the supremum in the definition of $T_j$ is achieved.
Since $\phi$ is a non-decreasing function, it is easy to check that
 \begin{equation}\label{tigx}
 T_{j}\leq T_{j+1}\ \ \ \text{and}\ \ \ \ T_{j+1}\leq T^{B+\varepsilon}_{j}.
 \end{equation}
We proceed to show that
\begin{equation}\label{tixjx}
 \liminf\limits_{n\to\infty}\frac{\log T_n}{\phi(n)}=1.
 \end{equation}
In fact, by the definition of $T_j$, we get $T_j \geq e^{\phi(j)}$ for all $j \geq 1$ and hence
\[
\liminf\limits_{n\to\infty}\frac{\log T_n}{\phi(n)}\geq1.
\]
We denote by $t_j \geq j$ the smallest number for which $T_j = e^{\phi(t_j)(B+\varepsilon)^{j-t_j}}$. Observe that for many consecutive j's, the number $t_j$ will be the same. More precisely, $t_j=t_{j+1}=\cdots=t_{t_j}$.
Let $\{\ell_k\}_{k\geq1}$ be the sequence of all $t_j's$ in the increasing order, without repetitions. For these $\ell_k$, we obtain $T_{\ell_k} = e^{\phi(\ell_k)}$, which gives
\[
\liminf\limits_{n\to\infty}\frac{\log T_n}{\phi(n)}\leq  \liminf\limits_{n\to\infty}\frac{\log T_{\ell_k} }{\phi(\ell_k)}  =1.
\]
Next we use $\{T_j\}_{j\geq1}$ to construct a suitable subset of ${E}_\phi$. Let $M$ be the positive integer such that $s_n=M\lfloor T_n\rfloor\geq3$ for all $n \geq 1$.
Write
\[
\mathbb{F}(T)=\big\{x\in [0,1): ns_n\leq a_n(x)<(n+1)s_n, \forall\ n\geq1\big\}.
\]
It is easy to verify that
\[\mathbb{F}(T) \subseteq E_\phi.\]
Since $\phi(n)/\log n \to\infty$ as $n \to \infty$, combining \eqref{tigx} with \eqref{tixjx}, we have
\[\limsup\limits_{n\to\infty}\frac{\log(n+1)}{\log T_n} =0\ \ \ \ \text{and}\ \ \ \log T_{n+1}-\log T_1\leq\big(B+\varepsilon-1\big)\sum\limits_{k=1}^{n}\log T_k.
\]
It follows from Lemma \ref{geshu} that
\[\dim_{\rm H}E_\phi \geq \dim_{\rm H}\mathbb{F}(T) =
\frac{1}{2+\xi},\] where $\xi$ is given by
\begin{align*}
\xi&= \limsup\limits_{n\to\infty}\frac{2\log(n+1)!+\log(M\lfloor T_{n+1}\rfloor)}{\sum\limits_{k=1}^{n}\log(M\lfloor T_k\rfloor)}\\
&\leq \limsup\limits_{n\to\infty}\frac{2\log(n+1)!}{\sum\limits_{k=1}^{n} \log T_k}+\limsup_{n\to\infty}\frac{\log T_{n+1}}{\sum\limits_{k=1}^{n}\log  T_k}\\
&\leq  \limsup\limits_{n\to\infty}\frac{2\log(n+1)}{\log T_n} +(B+\varepsilon-1)=B+\varepsilon-1.
\end{align*}
Hence \[\dim_{\rm H}E_\phi \geq \frac{1}{B+1+\varepsilon}.\] Letting $\varepsilon \to 0^+$ yields the assertion.

\begin{remark}
Let $\phi$ be a non-decreasing function and $\phi(n)/\log n \to\infty$ as $n \to \infty$.
From the proof of Theorem \ref{ybphi}, we conclude that
\[
\dim_{\rm H}\left\{x\in [0,1): \liminf\limits_{n\to\infty}\frac{\log a_n(x)}{\phi(n)}=1\right\}=\frac{1}{B+1}.
\]
\end{remark}

{\bf Acknowledgement:}
The authors would like to thank Professor Lingmin Liao for his invaluable comments. This research was supported by National Natural Science Foundation of China (11771153, 11801591) and Fundamental
Research Funds for the Central Universities SYSU-18lgpy65.

\end{document}